\documentclass[12pt]{amsart}

\usepackage{cite}                       
\usepackage{amsfonts,amssymb}           
\usepackage{bbm}                        
\usepackage[pdfstartview=FitH,
            CJKbookmarks=true,
            bookmarksnumbered=true,
            bookmarksopen=false,
            colorlinks,                 
            linkcolor=black,            
            anchorcolor=black,          
            citecolor=black             
            ]{hyperref}                 

\newtheorem{thm}{Theorem~}

\newtheorem{lem}[thm]{Lemma~}

\theoremstyle{remark}
\newtheorem{rem}[thm]{Remark~}

\newtheoremstyle{specthm}{1.5ex plus 1ex minus .2ex}{1.5ex plus 1ex minus .2ex}{\it}{}{\bf}{}{1em}{\thmnote{#3}}
\theoremstyle{specthm}

\def\R{{\mathbb R}}

\def\a{\alpha}

\def\d{\delta}

\def\l{\lambda}
\def\e{\varepsilon}

\newcommand{\lpa}{L^{1,p}(\R^n)}                        
\newcommand{\pap}[1]{P^{1,p}(\R^{#1})}                  
\newcommand{\sobo}[2]{W^{#1,#2}(\R^n)}                  
\newcommand{\lbv}{BV_{loc}(\R)}                         
\newcommand{\B}{\mathcal{B}_p}                          
\newcommand{\G}{\mathcal{G}_p}                          
\newcommand{\GL}[1]{\mathrm{GL}({#1})}                  
\newcommand{\SL}[1]{\mathrm{SL}({#1})}                  
\newcommand{\abs}[1]{\left\vert{#1}\right\vert}         
\newcommand{\absx}[1]{\vert{#1}\vert}                   
\newcommand{\set}[1]{\left\{{#1}\right\}}               
\newcommand{\setx}[1]{\{{#1}\}}                         
\newcommand{\inp}[1]{\left\langle{#1}\right\rangle}     

\newcommand{\cb}{\mathcal{K}^n}                         
\newcommand{\poly}[1]{\mathcal{P}_0^{#1}}               
\newcommand{\normk}[2]{\left\Vert{#2}\right\Vert_{#1}}  
\newcommand{\normkx}[2]{\Vert{#2}\Vert_{#1}}            
\newcommand{\fijk}{f_{i_{j_{k}}}}                       
\newcommand{\Dzlp}[1]{D^{#1}z_{\ell_P}(s)}

\begin{document}


\title[Real-valued valuations on Sobolev spaces]
{Real-valued valuations on Sobolev spaces}

\author{Dan Ma}

\address{Institut f\"{u}r Diskrete Mathematik und Geometrie\\
Technische Universit\"{a}t Wien\\
Wiedner Hauptstrasse 8--10/1046, Wien 1040, Austria}

\email{madan516@gmail.com}

\date{\today}


\begin{abstract}
    Continuous, $\SL{n}$ and translation invariant
    real-valued valuations on Sobolev spaces are classified.
\end{abstract}
\maketitle


\section{Introduction}

A function $z$ defined on a lattice $(\mathcal{L},\vee,\wedge)$
and taking values in an abelian semigroup is called a \emph{valuation} if
\begin{equation}\label{eqn:val}
    z(f\vee g)+z(f\wedge g)=z(f)+z(g)
\end{equation}
for all $f,g\in\mathcal{L}$. A function $z$ defined on some subset $\mathcal{M}$ of $\mathcal{L}$
is called a valuation on $\mathcal{M}$ if \eqref{eqn:val} holds whenever
$f,g,f\vee g,f\wedge g\in\mathcal{M}$.
Valuations were a key part of Dehn's solution of Hilbert's Third Problem in 1901.
They are closely related to dissections and lie at the very heart of geo\-metry.
Here, valuations were considered on the \mbox{space} of convex bodies
(i.e. compact convex sets) in $\R^n$, denoted by $\cb$.
Perhaps the most famous result is Hadwiger's \mbox{characterization} theorem
on this space which classifies all continuous and rigid \mbox{motion} invariant real-valued valuations.
Important later \mbox{contributions} can be found in \cite{Had57,KR97,McM93,McS83}.
As for recent results, we refer to
\cite{Ale99,Ale01,Hab09,Hab12a,Hab12b,HP14a,HL06,Kla96,Kla97,LYL15,Lud02b,Lud03,Lud06,LR99,LR10,Par14a,Par14b,SS06,Sch08,SW12,Wan11}.
For later reference, we state here a centro-affine version of
Hadwiger's characterization theorem on the space of convex polytopes
containing the origin in their interiors, which is denoted by $\poly{n}$.
\begin{thm}[\cite{HP14b}]\label{thm:HP}
    A map $Z:\poly{n}\rightarrow\R$ is an upper semicontinuous and $\SL{n}$ invariant valuation
    if and only if there exist constants $c_0,c_1,c_2\in\R$ such that
    \[Z(P)=c_0+c_1\abs{P}+c_2\abs{P^\ast}\]
    for all $P\in\poly{n}$,
    where $\abs{P}$ is the volume of $P$ and
    \[P^\ast=\set{x\in\R^n:\inp{x,y}\leq 1\text{ for all }y\in P}\]
    is the polar of $P$.
\end{thm}

Valuations are also considered on spaces of real-valued functions.
Here, we take the pointwise maximum and minimum as the join and meet, respectively.
Two important functions associated with
every convex body $K$ in $\R^n$ are the indicator function $\mathbbm{1}_K$
and the support function $h(K,\cdot)$, where $h(K,u)=\max\set{\inp{u,x}:x\in K}$
and $\inp{u,x}$ is the standard inner product of $u,x\in\R^n$.
As each of them is in one-to-one correspondence with $K$,
valuations on these function spaces are often considered
to be valuations on convex bodies.

Valuations on other classical function spaces have been characterized since 2010.
Tsang \cite{TSa10} characterized real-valued valuations on $L^p$-spaces.
\begin{thm}[\cite{TSa10}]\label{thm:Tsang}
    A functional $z:L^p(\R^n)\rightarrow\R$ is a conti\-nuous translation invariant valuation
    if and only if there exists a conti\-nuous function on $\R$ with the property that
    there exists $c\geq0$ such that $\abs{h(x)}\leq c\abs{x}^p$ for all $x\in\R$ and
    \[z(f)=\int_{\R^n}h\circ f\]
    for every $f\in L^p(\R^n)$.
\end{thm}

Kone \cite{Kon14} generalized this characterization to Orlicz spaces.
As for valuations on Sobolev spaces, Ludwig characterized
the Fisher information matrix and the optimal Sobolev body.
Throughout this paper, the Sobolev space on $\R^n$ with indices $k$ and $p$
is denoted by $\sobo{k}{p}$ (see Section 2 for precise definitions)
and the additive group of \mbox{real} symmetric $n\times n$ matrices
is denoted by $\langle\mathbb{M}^n,+\rangle$.
An operator $z:\sobo{1}{2}\rightarrow\langle\mathbb{M}^n,+\rangle$
is called \emph{$\GL{n}$ contravariant} if for some $p\in\R$,
\[z\left(f\circ\phi^{-1}\right)=\abs{\det\phi}^p\phi^{-t}z(f)\phi^{-1}\]
for all $f\in\sobo{1}{2}$ and $\phi\in\GL{n}$,
where $\det\phi$ is the determinant of $\phi$
and $\phi^{-t}$ denotes the inverse of the transpose of $\phi$.
An operator $z:\sobo{1}{2}\rightarrow\langle\mathbb{M}^n,+\rangle$
is called \emph{affinely contravariant} if it is $\GL{n}$ contravariant,
translation invariant, and homogeneous (see Section 2 for precise definitions).
\begin{thm}[\cite{Lud11b}]\label{thm:fisher}
    An operator $z:\sobo{1}{2}\rightarrow\langle\mathbb{M}^n,+\rangle$, where $n\geq3$,
    is a continuous and affinely contravariant valuation if and only if
    there is a constant $c\in\R$ such that
    \[z(f)=c\int_{\R^n}\nabla f\otimes\nabla f\]
    for every $f\in\sobo{1}{2}$.
\end{thm}

Other recent and interesting characterizations can be found in
\cite{CC15,Tsa12,BGW13,Lud13,Wang14,Obe14}.

In this paper, we classify real-valued valuations on $\sobo{1}{p}$.
The result regarding homogeneous valuations is stated first.
Let $1\leq p<n$ throughout this paper.
Furthermore, we say that a valuation is \emph{trivial} if it is identically zero.
\begin{thm}\label{thm:homo}
    A functional $z:\sobo{1}{p}\rightarrow\R$ is a non-trivial conti\-nuous,
    $\SL{n}$ and translation invariant valuation
    that is homogeneous of degree $q$ if and only if $p\leq q\leq \frac{np}{n-p}$
    and there exists a constant $c\in\R$ such that
    \begin{equation}\label{eqn:homo}
        z(f)=c\normk{q}{f}^q
    \end{equation}
    for every $f\in\sobo{1}{p}$.
\end{thm}

It is natural to consider the same characterization without the assump\-tion of homogeneity.
It turns out to be more complicated and costs addi\-tional assumptions.
We first fix the following notation.
Let $C^k(\R^n)$ denote the space of functions on $\R^n$
that have $k$ times continuous partial derivatives
for a positive integer $k$;
let $\lbv$ denote the space of functions on $\R$
that are of locally bounded variation.
We denote by $\G$ the class of functions $g$
that belong to $\lbv$ and satisfy
\begin{equation}\label{eqn:gc}
    g(x)\sim\left\{\begin{array}{ll}
                    O(x^p), & \text{ as }x\rightarrow0; \\
                    O(x^{\frac{np}{n-p}}), & \text{ as }x\rightarrow\infty.
                  \end{array}
    \right.
\end{equation}
and by $\B$ the class of functions $g$
that belong to $C^n(\R)$ with $g^{(n)}\in\lbv$
and $x^kg^{(k)}(x)$ satisfying \eqref{eqn:gc} for each integer $1\leq k\leq n$.
Let $\pap{n}$ be the set of functions $\ell_P$ with $P\in\poly{n}$
that enclose pyramids of height 1 on $P$ (see Section 2 for the precise definition).
\begin{thm}\label{thm:main}
    A functional $z:\sobo{1}{p}\rightarrow\R$ is a continuous, $\SL{n}$
    and translation invariant valuation with $z(0)=0$ and
    $s\mapsto z(sf)$ in $\B$ for $s\in\R$ and $f\in\pap{n}$ if and only if
    there exists a continuous function $h\in\G$ such that
    \[z(f)=\int_{\R^n}h\circ f,\]
    for every $f\in\sobo{1}{p}$.
\end{thm}

The proofs of Theorem \ref{thm:homo} and \ref{thm:main} can be found
in Sections \ref{sec:homo} and \ref{sec:final}, respectively.


\section{Preliminaries}\label{sec:P}

For $p\geq1$ and a measurable function $f:\R^n\rightarrow\R$, let
\[\normk{p}{f}=\left(\int_{\R^n}\abs{f(x)}^pdx\right)^{1/p}.\]
Define $L^p(\R^n)$ to be the class of measurable functions with $\normk{p}{f}<\infty$
and $L_{loc}^p(\R^n)$ to be the class of measurable functions
with $\normk{p}{f\mathbbm{1}_K}<\infty$ for every compact $K\subset\R^n$.

A measurable function $\nabla f:\R^n\rightarrow\R^n$ is said
to be the \emph{weak gradient} of $f\in L^p(\R^n)$ if
\begin{equation}\label{eqn:delta}
    \int_{\R^n}\nu(x)\cdot\nabla f(x)dx=-\int_{\R^n}f(x)\nabla\cdot\nu(x)dx
\end{equation}
for every compactly supported smooth vector field $\nu:\R^n\rightarrow\R^n$,
where $\nabla\cdot\nu=\frac{\partial\nu_1}{\partial x_1}+\cdots+\frac{\partial\nu_n}{\partial x_n}$.
A function $f\in L^1(\R^n)$ is said to be of \emph{bounded variation} on $\R^n$ if
there exists a finite signed vector-valued Radon measure $\l$ on $\R^n$
such that
\[\int_{\R^n}\nu(x)\cdot\nabla f(x)dx=\int_{\R^n}\nu(x)\cdot d\l(x)\]
for every $\nu$ as mentioned before.
A function $f\in L_{loc}^1(\R^n)$ is said to be of \emph{locally bounded variation} on $\R^n$
if $f$ is of bounded variation on all open subset of $\R^n$.

The \emph{Sobolev space} $\sobo{1}{p}$ consists of all functions $f\in L^p(\R^n)$
whose weak gradient belongs to $L^p(\R^n)$ as well.
For each $f\in\sobo{1}{p}$, we define the \emph{Sobolev norm} to be
\[\normk{\sobo{1}{p}}{f}=\left(\normk{p}{f}^p+\normk{p}{\nabla f}^p\right)^{1/p},\]
where $\normk{p}{\nabla f}$ denotes the $L^p$ norm of $\abs{\nabla f}$.
Equipped with the Sobolev norm, the Sobolev space $\sobo{1}{p}$ is a Banach space.
\begin{thm}[\cite{LL01}]\label{thm:cvg}
    Let $\setx{f_i}$ be a sequence in $\sobo{1}{p}$ that
    converges to $f\in\sobo{1}{p}$. Then, there exists
    a subsequence $\setx{f_{i_j}}$ that converges to
    $f$ a.e. as $j\rightarrow\infty$.
\end{thm}
Furthermore, for $1\leq p<n$, $\sobo{1}{p}$ is continuously embedded in $L^q(\R^n)$
for all $p\leq q\leq p^\ast$, where $p^\ast=\frac{np}{n-p}$ is the \emph{Sobolev conjugate} of $p$,
due to the Sobolev-Gagliardo-Nirenberg inequality stated as the following theorem.
\begin{thm}[\cite{Leo09}]\label{thm:SGN}
    Let $1\leq p<n$. There exists a positive constant $C$,
    depending only on $p$ and $n$, such that
    \[\normk{p^\ast}{f}\leq C\normk{p}{\nabla f}\]
    for all $f\in\sobo{1}{p}$.
\end{thm}
\begin{rem}
    By Theorem \ref{thm:SGN}, the expression in (\ref{eqn:homo}) is well defined.
\end{rem}

For $f,g\in\sobo{1}{p}$, we have $f\vee g$, $f\wedge g\in\sobo{1}{p}$ and for almost every $x\in\R^n$,
\begin{equation*}
    \nabla(f\vee g)(x)=\left\{
                         \begin{array}{ll}
                           \nabla f(x), & \text{when }f(x)>g(x) \\
                           \nabla g(x), & \text{when }f(x)<g(x) \\
                           \nabla f(x)=\nabla g(x), & \text{when }f(x)=g(x)
                         \end{array}
                       \right.
\end{equation*}
and
\begin{equation*}
    \nabla(f\wedge g)(x)=\left\{
                         \begin{array}{ll}
                           \nabla f(x), & \text{when }f(x)<g(x) \\
                           \nabla g(x), & \text{when }f(x)>g(x) \\
                           \nabla f(x)=\nabla g(x), & \text{when }f(x)=g(x)
                         \end{array}
                       \right.
\end{equation*}
(see \cite{LL01}). Hence $(\sobo{1}{p},\vee,\wedge)$ is a lattice.

Let $\lpa\subset\sobo{1}{p}$ be the space of piecewise affine functions on $\R^n$.
Here, a function $\ell:\R^n\rightarrow\R$ is called \emph{piecewise affine},
if it is continuous and there exists a finite number of $n$-dimensional simplices
$\triangle_1,\ldots,\triangle_m\subset\R^n$ with pairwise disjoint interiors
such that the restriction of $\ell$ to each $\triangle_i$ is affine
and $\ell=0$ outside $\triangle_1\cup\cdots\cup\triangle_m$.
The simplices $\triangle_1,\ldots,\triangle_m$ are called a triangulation
of the support of $\ell$. Let $V$ denote the set of vertices of this triangulation.
We further have that $V$ and the values $\ell(v)$ for $v\in V$ completely determine $\ell$.
Piecewise affine functions lie dense in $\sobo{1}{p}$ (see \cite{Leo09}).

For $P\in\poly{n}$, define the piecewise affine function $\ell_P$ by requiring that $\ell_P(0)=1$,
$\ell_P(x)=0$ for $x\notin P$, and $\ell_P$ is affine on each simplex with apex at the origin
and base among facets of $P$. Define $\pap{n}\subset\lpa$ as the set of all $\ell_P$
for $P\in\poly{n}$. For $\phi\in\GL{n}$, we have $\ell_{\phi P}=\ell_P\circ\phi^{-1}$.
We remark that multiples and translates of $\ell_P\in\pap{n}$
correspond to linear elements within the theory of finite elements.

For $P\in\poly{n}$, let $F_1,\ldots,F_m$ be the facets of $P$.
For each facet $F_i$, let $u_i$ be its unit outer normal vector
and $T_i$ the convex hull of $F_i$ and the origin.
For $x\in T_i$, notice that
\[\ell_P(x)=-\inp{\frac{u_i}{h(P,u_i)},x}+1\]
and
\[\nabla\ell_P(x)=-\frac{u_i}{h(P,u_i)}.\]
It follows that
\begin{eqnarray*}
    \normk{p}{\ell_P}^p &=& \int_{\R^n}\abs{\ell_P}^pdx\\
    &=& p\int_0^1t^{p-1}\abs{\set{\ell_P>t}}dt\\
    &=& p\abs{P}\int_0^1t^{p-1}(1-t)^ndt\\
    &=& c_{p,n}\abs{P},
\end{eqnarray*}
where $c_{p,n}=\frac{\Gamma(p+1)\Gamma(n+1)}{\Gamma(n+p+1)}={{n+p}\choose n}^{-1}$, and
\begin{eqnarray*}
    \normk{p}{\nabla\ell_P}^p &=& \int_{\R^n}\abs{\nabla\ell_P(x)}^pdx\\
    &=& \sum_{i=1}^m\int_{T_i}\abs{\frac{u_i}{h(P,u_i)}}^pdx\\
    &=& \sum_{i=1}^m\frac{\abs{T_i}}{h^p(P,u_i)}\\
    &=& \frac{1}{n}\sum_{i=1}^m\abs{F_i}h^{1-p}(P,u_i)\\
    &=& \frac{1}{n}S_p(P),
\end{eqnarray*}
where $S_p(P)$ is the $p$-surface area of $P$.

Let $z:\sobo{1}{p}\rightarrow\R$ be a functional.
The functional is called \emph{continuous} if for every sequence $f_k\in\sobo{1}{p}$
with $f_k\rightarrow f$ as $k\rightarrow\infty$ with respect to the Sobolev norm,
we have $\abs{z(f_k)-z(f)}\rightarrow0$ as $k\rightarrow\infty$.
It is said to be \emph{translation invariant} if $z(f\circ\tau^{-1})=z(f)$ for all $f\in\sobo{1}{p}$
and translations $\tau$. Furthermore, we say it is homogeneous if for some $q\in\R$,
we have $z(sf)=\abs{s}^qz(f)$ for all $f\in\sobo{1}{p}$ and $s\in\R$.
Finally, we call the functional \emph{$\SL{n}$ invariant} if $z(f\circ\phi^{-1})=z(f)$
for all $f\in\sobo{1}{p}$ and $\phi\in\SL{n}$. Denote the derivative
of the map $s\mapsto z(sf)$ by
\[Dz_f(s)=\lim_{\e\rightarrow0}\frac{z\left((s+\e)f\right)-z(sf)}{\e}\]
whenever it exists.

We provide a set of examples of valuations on $\sobo{1}{p}$ in the following theorem.
\begin{thm}\label{thm:phi}
    Let $h\in\G$ be a continuous function.
    Then, for every $f\in\sobo{1}{p}$, the functional
    \[z(f)=\int_{\R^n}h\circ f\]
    is a continuous, $\SL{n}$ and translation invariant valuation.
    Furthermore, $z(0)=0$ and the map $s\mapsto z(sf)$ belongs to $\B$
    for every $s\in\R$ and $f\in\pap{n}$.
\end{thm}
\begin{proof}
    We can easily prove this theorem by verifying the following distinguishable
    properties of the functional.

    1. Valuation.
    Let $f,g\in\sobo{1}{p}$ and $E=\set{x\in\R^n:f(x)\geq g(x)}$. Then
    \begin{eqnarray*}
        z(f\vee g)+z(f\wedge g) &=& \int_{\R^n}h\circ(f\vee g)+\int_{\R^n}h\circ(f\wedge g)\\
        &=& \int_Eh\circ(f\vee g)+\int_{\R^n\setminus E}h\circ(f\vee g)\\
        &&+\int_Eh\circ(f\wedge g)+\int_{\R^n\setminus E}h\circ(f\wedge g)\\
        &=& \int_Eh\circ f+\int_{\R^n\setminus E}h\circ g+\int_Eh\circ g+\int_{\R^n\setminus E}h\circ f\\
        &=& \int_{\R^n}h\circ f+\int_{\R^n}h\circ g\\
        &=& z(f)+z(g).
    \end{eqnarray*}

    2. Continuity.
    Let $f\in\sobo{1}{p}$ and $\set{f_i}$ be a sequence in $\sobo{1}{p}$
    that converges to $f$. For every subsequence $\set{z(f_{i_j})}\subset\set{z(f_i)}$,
    we are going to show that there exists a subsequence $\setx{z(\fijk)}$
    that converges to $z(f)$. Let $\setx{f_{i_j}}$ be a subsequence of $\set{f_i}$.
    Then, $\setx{f_{i_j}}$ converges to $f$ in $\sobo{1}{p}$.
    Thus, there exists a subsequence $\setx{\fijk}\subset\setx{f_{i_j}}$ with
    $\fijk\rightarrow f$ a.e. as $k\rightarrow\infty$.
    Furthermore, since $h$ is continuous, we obtain $h\circ\fijk\rightarrow h\circ f$
    a.e. as $k\rightarrow\infty$. Since $h$ satisfies \eqref{eqn:gc},
    there exist $\d>0$ and $M_1>0$ such that $\abs{h(x)}\leq M_1\abs{x}^p$ when $\abs{x}<\d$.
    Let $E_1=\set{\abs{f}<3\d/4}$. Since $\fijk\rightarrow f$ a.e. as $k\rightarrow\infty$,
    for such $\d>0$, there exists $N_1>0$ such that $\absx{\fijk-f}<\d/4$ a.e. whenever $k>N_1$.
    Thus, $\absx{\fijk}<\d$ a.e. on $E_1$.
    Hence, for such $k$, $\absx{h\circ\fijk}\leq M_1\absx{\fijk}^p$ a.e. on $E_1$.
    \mbox{Since} $M_1\int_{E_1}\absx{\fijk}^p\leq M_1\normkx{p}{\fijk}^p<\infty$,
    by the dominated convergence theorem, we have
    \[\lim_{k\rightarrow\infty}\int_{E_1}h\circ\fijk=\int_{E_1}h\circ f.\]
    On the other hand, there exist $M_0>0$ and $M_2>0$
    such that, whenever $\abs{x}>M_0$, we obtain $\abs{h(x)}\leq M_2\abs{x}^{p^\ast}$.
    Let $E_2=\set{\abs{f}>3M_0/2}$. Since $\fijk\rightarrow f$ a.e. as $k\rightarrow\infty$,
    there exists $N_2>0$ for such $M_0>0$ such that $\absx{\fijk-f}<M_0/2$ a.e. whenever $k>N_2$.
    Thus, $\absx{\fijk}>M_0$ a.e. on $E_2$.
    Hence, for such $k$, $\absx{h\circ\fijk}\leq M_2\absx{\fijk}^{p^\ast}$ a.e. on $E_2$.
    Since
    \[M_2\int_{E_2}\abs{\fijk}^{p^\ast}\leq M_2\normk{p^\ast}{\fijk}^{p^\ast}\leq CM_2\normk{p}{\nabla\fijk}^{p^\ast}<\infty,\]
    by the dominated convergence theorem, we obtain
    \[\lim_{k\rightarrow\infty}\int_{E_2}h\circ\fijk=\int_{E_2}h\circ f.\]
    Now let $E_3=\R^n\setminus(E_1\cup E_2)$ and $N=\max\set{N_1,N_2}$.
    Then, for $k>N$, we have $\d/2\leq\absx{\fijk}\leq 2M_0$ a.e. on $E_3$.
    Thus, for such $k$, since $h$ is continuous, there exists $\gamma>0$
    such that $\absx{h\circ\fijk}\leq\gamma\absx{\fijk}$ a.e. on $E_3$. Since
    \[\gamma\int_{E_3}\abs{\fijk}\leq\gamma\normk{1}{\fijk}\leq\gamma\normk{p}{\fijk}<\infty,\]
    again by the dominated convergence theorem, we obtain
    \[\lim_{k\rightarrow\infty}\int_{E_3}h\circ\fijk=\int_{E_3}h\circ f.\]

    3. $\SL{n}$ invariance.
    Let $f\in\sobo{1}{p}$ and $\phi\in\SL{n}$. Then
    \[z(f\circ\phi^{-1})=\int_{\R^n}h\circ f\circ\phi^{-1}=\int_{\R^n}h\left(f\left(\phi^{-1}x\right)\right)dx.\]
    By setting $y=\phi^{-1}x$, we obtain
    \begin{eqnarray*}
        z(f\circ\phi^{-1}) &=& \int_{\R^n}h\left(f(y)\right)dy\\
        &=& \int_{\R^n}h\circ f=z(f).
    \end{eqnarray*}

    4. Translation invariance.
    Let $f\in\sobo{1}{p}$ and $\tau$ be a translation. Then
    \[z(f\circ\tau^{-1})=\int_{\R^n}h\circ f\circ\tau^{-1}=\int_{\R^n}h\left(f\left(\tau^{-1}x\right)\right)dx.\]
    By setting $y=\tau^{-1}x$, we obtain
    \begin{eqnarray*}
        z(f\circ\tau^{-1}) &=& \int_{\R^n}h\left(f(y)\right)dy\\
        &=& \int_{\R^n}h\circ f=z(f).
    \end{eqnarray*}

    5. $z(0)=0$. This fact follows from the continuity of $z$ and \eqref{eqn:gc}.

    6. Differentiability.
    Let $\ell_P\in\pap{n}$ where $P\in\poly{n}$. Without loss of generality,
    we assume $s>0$. Indeed, set
    \[h^e(x)=\frac{h(x)+h(-x)}{2}\text{ and }h^o(x)=\frac{h(x)-h(-x)}{2},\]
    for every $x\in\R$, and the case $s<0$ follows from
    \begin{eqnarray*}
        z(-s\ell_P) &=& \int_{\R^n}(h^e+h^o)\circ(-s\ell_P)\\
        &=& \int_{\R^n}\left(h^e\circ(-s\ell_P)+h^o\circ(-s\ell_P)\right)\\
        &=& \int_{\R^n}\left(h^e\circ(s\ell_P)-h^o\circ(s\ell_P)\right).
    \end{eqnarray*}
    Since $h\in\lbv$, there exists a signed measure $\nu$ on $\R$
    such that $h(s)=\nu([0,s))$ for every $s>0$
    (this can be done by setting $\nu=\mathbbm{1}_{[0,s)}$ in \eqref{eqn:delta}).
    By the layer cake representation, we have
    \begin{eqnarray*}
        z(s\ell_P) &=& \int_{\R^n}h\circ(s\ell_P)\\
        &=& \int_0^s\abs{\set{s\ell_P>t}}d\nu(t)=\abs{P}\int_0^s\left(\frac{s-t}{s}\right)^nd\nu(t).
    \end{eqnarray*}
    In other words, we obtain
    \begin{equation}\label{eqn:fune}
        s^nz(s\ell_P)=\abs{P}\int_0^s(s-t)^nd\nu(t).
    \end{equation}
    We will now show the differentiability by induction.
    Let $k\geq2$ and $\psi_k(s)$ be the $k$th derivative of
    $\int_0^s(s-t)^nd\nu(t)$ with respect to $s$. We have
    \begin{equation}\label{eqn:diff1}
        \psi_k(s)=\frac{n!}{(n-k)!}\int_0^s(s-t)^{n-k}d\nu(t).
    \end{equation}
    In particular, we obtain $\psi_n(s)=n!h(s)$.
    On the other hand, differentiating the left hand side of \eqref{eqn:fune} gives
    \[\psi_1(s)\abs{P}=ns^{n-1}z(s\ell_P)+s^n\Dzlp{}.\]
    By induction, it follows that
    \begin{equation}\label{eqn:diff2}
        \psi_k(s)\abs{P} = \sum_{j=0}^k{k\choose j}\frac{n!}{(n-k+j)!}s^{n-k+j}\Dzlp{j}.
    \end{equation}
    In particular, we obtain
    \[\psi_n(s)\abs{P}=n!\sum_{j=0}^n{n\choose j}s^j\Dzlp{j},\]
    which coincides with $n!\abs{P}h(s)$.
    Since $h$ is a continuous locally $BV$ function,
    we have the desired differentiability of $s\mapsto z(s\ell_P)$.

    7. Growth condition.
    First of all, by \eqref{eqn:fune},
    \begin{eqnarray*}
        z(s\ell_P) &=& \abs{P}\int_0^s\left(\frac{s-t}{s}\right)^nd\nu(t)\\
        &\leq& \abs{P}\int_0^sd\nu(t)\\
        &=& \abs{P}h(s)
    \end{eqnarray*}
    satisfies \eqref{eqn:gc}. As shown in the previous steps
    (\eqref{eqn:diff1} and \eqref{eqn:diff2}),
    for every integer $1\leq k\leq n$,
    \[\sum_{j=0}^k{k\choose j}\frac{n!}{(n-k+j)!}s^{n-k+j}\Dzlp{j}=\frac{n!}{(n-k)!}\abs{P}\int_0^s(s-t)^{n-k}d\nu(t).\]
    That is
    \begin{eqnarray*}
        \sum_{j=0}^k{k\choose j}\frac{n!}{(n-k+j)!}s^j\Dzlp{j} &=& \frac{n!}{(n-k)!}\abs{P}\int_0^s\left(\frac{s-t}{s}\right)^{n-k}d\nu(t)\\
        &\leq& \frac{n!}{(n-k)!}\abs{P}h(s)
    \end{eqnarray*}
    also satisfies \eqref{eqn:gc}.
\end{proof}


\section{The characterization of homogeneous valuations}\label{sec:homo}

First, we need the following reduction similar to \cite[Lemma 8]{Lud12}.
We include the proof for the sake of completeness.

\begin{lem}\label{thm:L1p}
    Let $z_1,z_2:\lpa\rightarrow\R$ be continuous and translation invariant valuations
    satisfying $z_1(0)=z_2(0)=0$. If $z_1(sf)=z_2(sf)$ for all $s\in\R$ and $f\in\pap{n}$, then
    \begin{equation}\label{eqn:equal}
        z_1(f)=z_2(f)
    \end{equation}
    for all $f\in \lpa.$
\end{lem}
\begin{proof}
    We first make the following four reduction steps for \eqref{eqn:equal}.

    1. We begin by considering all $f\in \lpa$ that are non-negative.
    Since $z_1$ and $z_2$ are valuations
    satisfying $z_1(0)=z_2(0)=0$, we have for $i=1,2$,
    \[z_i(f\vee0)+z_i(f\wedge0)=z_i(f)+z_i(0)=z_i(f).\]
    For $i=1,2$, let
    \[z_i^e(f)=\frac{z_i(f)+z_i(-f)}{2},z_i^o(f)=\frac{z_i(f)-z_i(-f)}{2}\]
    and hence $z_i(f)=z_i^e(f)+z_i^o(f)$ for all $f\in \lpa$.
    Therefore, we have
    \[z_i^e(f\wedge0)=z_i^e\left(-((-f)\wedge0)\right)=z_i^e\left((-f)\wedge0\right)\]
    and
    \[z_i^o(f\wedge0)=z_i^o\left(-((-f)\wedge0)\right)=-z_i^o\left((-f)\wedge0\right).\]
    Thus, it suffices to show that \eqref{eqn:equal} holds for all non-negative $f\in\lpa$.

    2. Next, let $f\in \lpa$ where the values $f(v)$ are distinct for $v\in V$ with $f(v)>0$.
    Let $f$ not vanish identically and $\mathcal{S}$ be the triangulation
    of the support of $f$ in $n$-dimensional simplices
    such that $f\big|_\triangle$ is affine for each simplex $\triangle\in\mathcal{S}$.
    Denote by $V$ the (finite) set of vertices of $\mathcal{S}$.
    Note that $f$ is determined by its value on $V$.
    By the continuity of $z_1$ and $z_2$, we have the reduction as
    there always exists an approximation of $f$
    by $g\in \lpa$ where the values $g(v)$ are distinct for $v\in V$ with $g(v)>0$.

    3. We now consider all $f\in \lpa$ that are concave on their supports.
    Let $f_1,\ldots,f_m\in \lpa$ be non-negative
    and concave on their supports such that
    \begin{equation}\label{eqn:concave}
        f=f_1\vee\cdots\vee f_m.
    \end{equation}
    For $i=1,2$, by the inclusion-exclusion principle, we obtain
    \[z_i(f)=z_i(f_1\vee\cdots\vee f_m)=\sum_J(-1)^{\abs{J}-1}z_i(f_J),\]
    where $J$ is a non-empty subset of $\set{1,\ldots,m}$ and
    \[f_J=f_{j_1}\wedge\cdots\wedge f_{j_k},\]
    for $J=\set{j_1,\ldots,j_k}$. Indeed, such representation in \eqref{eqn:concave} exists.
    We determine the $f_i$'s by their value on $V$.
    Set $f_i(v)=f(v)$ on the vertices $v$ of the simplex $\triangle_i$ of $\mathcal{S}$.
    Choose a polytope $P_i$ containing $\triangle_i$ and set $f_i(v)=0$
    on the vertices $v$ of $P_i$. If the $P_i$'s are chosen suitably small, \eqref{eqn:concave} holds.
    The reduction follows since the meet of concave functions
    is still concave.

    4. Then take all functions $f\in \lpa$ such that $F$ defined below, is not singular.
    Given a function $f\in \lpa$,
    let $F\subset\R^{n+1}$ be the compact polytope bounded
    by the graph of $f$ and the hyperplane $\set{x_{n+1}=0}$.
    We say $F$ is \emph{singular} if $F$ has $n$ facet hyperplanes that intersect in a line $L$
    parallel to $\set{x_{n+1}=0}$ but not contained in $\set{x_{n+1}=0}$.
    Similar to the second step, by continuity of $z_1$ and $z_2$,
    it suffices to show \eqref{eqn:equal} for $f\in \lpa$
    such that $F$ is not singular.

    Let a function $f$ satisfying reduction steps 1-4 be given.
    Denote by $\bar{p}$ the vertex of $F$ with the largest $x_{n+1}$-coordinate.
    We are now going to show \eqref{eqn:equal} by induction on
    the number $m$ of facet hyperplanes of $F$ that are not passing through $\bar{p}$.
    In the case $m=1$, a scaled translate of $f$ is in $\pap{n}$.
    Since $z_1$ and $z_2$ are translation invariant, equation \eqref{eqn:equal} holds.
    Let $m\geq2$. Further let $p_0=\left(x_0,f(x_0)\right)$ be a vertex of $F$
    with minimal $x_{n+1}$-coordinate
    and $H_1,\ldots,H_j$ be the facet hyperplanes of $F$ through $p_0$
    which do not contain $\bar{p}$. Notice that there exists at least one such hyperplane.
    Write $\bar{F}$ as the polytope bounded by the intersection of
    all facet hyperplanes of $F$ other than $H_1,\ldots,H_j$.
    Since $F$ is not singular, $\bar{F}$ is bounded.
    Thus, there exists an $f\in \lpa$ that corresponds to $F$.
    Note that $\bar{F}$ has at most $(m-1)$ facet hyperplanes not containing $\bar{p}$.
    Let $\bar{H}_1,\ldots,\bar{H}_i$ be the facet hyperplanes of $\bar{F}$ that contain $p_0$.
    Choose hyperplanes $\bar{H}_{i+1},\ldots,\bar{H}_k$ also containing $p_0$
    such that the hyperplanes $\bar{H}_1,\ldots,\bar{H}_k$ and $\set{x_{n+1}=0}$
    enclose a pyramid with apex at $p_0$ that is contained in $\bar{F}$
    and has $x_0$ in its base with $\bar{H}_1,\ldots,\bar{H}_i$ among its facet hyperplanes.
    Therefore, there exists a piecewise affine function $\ell$
    corresponding to this pyramid. Moreover, a scaled translate of $\ell$ is in $\pap{n}$.
    We also obtain that a scaled translate of $\bar{\ell}=f\wedge\ell$ is in $\pap{n}$.
    To summarize, scaled translates of $\bar{\ell}$ and $\ell$ are in $\pap{n}$,
    the polytope $\bar{F}$ has at most $(m-1)$ facet hyperplanes not containing $\bar{p}$, and
    \[f\vee\ell=\bar{f}\text{ and }f\wedge\ell=\bar{\ell}.\]
    Applying valuations $z_1$ and $z_2$, we have, for $i=1,2$,
    \[z_i(f)+z_i(\ell)=z_i(\bar{f})+z_i(\bar{\ell}).\]
    Thus, the induction hypotheses yields the desired result.
\end{proof}


The classification will also make use of the following elementary fact.

\begin{rem}\label{thm:bigo}
    Let $f$ and $g$ be functions on $\R$.
    If $f(x)\sim o\left(g(x)h(x)\right)$ as $x\rightarrow0$,
    for each function $h$ on $\R$ with $\lim_{x\rightarrow0}h(x)=\infty$,
    then
    \[f(x)\sim O(g(x))\text{ as }x\rightarrow0.\]
    This can be seen by the following simple argument.
    Suppose $\abs{f(x)\big/g(x)}\rightarrow\infty$ as $x\rightarrow0$.
    Let $h=\sqrt{\abs{f\big/g}}$. It is clear that $h(x)\rightarrow\infty$
    as $x\rightarrow0$. But now
    \[\abs{f(x)\big/\left(g(x)h(x)\right)}=\sqrt{\abs{f(x)\big/g(x)}}=h(x)\rightarrow\infty\text{ as }x\rightarrow0,\]
    which yields a contradiction.
    A similar argument also works for the limit as $x\rightarrow\infty$.
\end{rem}


\begin{lem}\label{thm:P1p}
    Let $z:\lpa\rightarrow\R$ be a continuous, $\SL{n}$
    and translation invariant valuation with $z(0)=0$.
    Then there exists a continuous function $c:\R\rightarrow\R$
    satisfying \eqref{eqn:gc} such that
    \[z(s\ell_P)=c(s)\abs{P},\]
    for every $s\in\R$ and $\ell_P\in\pap{n}$.
\end{lem}

\begin{proof}
    Similar to the proof of Lemma 5 in \cite{Lud12}.
    Define the functional $Z:\poly{n}\rightarrow\R$ by setting
    \[Z(P)=z(s\ell_P),\]
    for every $s\in\R$ and $\ell_P\in\pap{n}$.
    If $\ell_P,\ell_Q\in\pap{n}$ are such that $\ell_P\vee\ell_Q\in\pap{n}$, then $\ell_P\vee\ell_Q=\ell_{P\cup Q}$ and
    $\ell_P\wedge\ell_Q=\ell_{P\cap Q}$. Since $z$ is a valuation on $\lpa$, it follows that
    \begin{eqnarray*}
        Z(P)+Z(Q) &=& z(s\ell_P)+z(s\ell_Q)\\
        &=& z(s(\ell_P\vee\ell_Q))+z(s(\ell_P\wedge\ell_Q))\\
        &=& Z(P\cup Q)+Z(P\cap Q),
    \end{eqnarray*}
    for $P,Q,P\cup Q\in\poly{n}$.
    Thus, $Z:\poly{n}\rightarrow\R$ is a valuation.

    By Theorem \ref{thm:HP}, there exist $c_0,c_1,c_2\in\R$ depending now on $s$ such that
    \begin{equation}\label{eqn:hp}
        z(s\ell_P)=c_0(s)+c_1(s)\abs{P}+c_2(s)\abs{P^\ast},
    \end{equation}
    for all  $s\in\R$ and $\ell_P\in\pap{n}$.
    We now investigate the behavior of these constants by studying valuations
    on different $s\ell_P$'s and their translations, for $s\in\R$.

    We start with $c_0$ and $c_2$.

    \textbf{Example 1.}
    Let $P\in\poly{n}$. Take translations $\tau_1,\ldots,\tau_k$ such that
    the $\phi_iP$'s are pairwise disjoint, where $\phi_iP=\tau_i\left(P/k^i\right)$.
    Consider the function $f_k=s\left(\ell_{\phi_1P}\vee\cdots\vee\ell_{\phi_kP}\right)$, $s\in\R$.
    Then, we have
    \begin{eqnarray*}
        \normk{p}{f_k}^p &=& \abs{s}^p\sum_{i=1}^k\int_{\phi_iP}\ell_{\phi_iP}^p=\abs{s}^p\sum_{i=1}^k\int_{\phi_iP}\left(\ell_P\left(\phi_i^{-1}x\right)\right)^pdx\\
        &=& \abs{s}^p\sum_{i=1}^kk^{-in}\int_P\ell_P^p=\abs{s}^p\normk{p}{\ell_P}^p\sum_{i=1}^kk^{-in}\rightarrow0\mbox{ as }k\rightarrow\infty.
    \end{eqnarray*}
    Further,
    \begin{eqnarray*}
        \normk{p}{\nabla f_k}^p &=& \abs{s}^p\sum_{i=1}^k\int_{\phi_iP}\abs{\nabla\ell_{\phi_iP}}^p=\abs{s}^p\sum_{i=1}^k\int_{\phi_iP}\abs{\nabla\left(\ell_P\circ\phi_i^{-1}\right)}^p\\
        &=& \abs{s}^p\sum_{i=1}^k\int_{\phi_iP}\abs{\phi_i^{-t}\nabla\ell_P\left(\phi_i^{-1}x\right)}^pdx=\abs{s}^p\sum_{i=1}^kk^{ip}\int_{\phi_iP}\abs{\nabla\ell_P\left(\phi_i^{-1}x\right)}^pdx\\
        &=& \abs{s}^p\sum_{i=1}^kk^{-i(n-p)}\normk{p}{\nabla\ell_P}^p\rightarrow0\mbox{ as }k\rightarrow\infty.
    \end{eqnarray*}
    Thus, $f_k\rightarrow0$ in $\sobo{1}{p}$ as $k\rightarrow\infty$.

    By the translation invariance of $z$ and \eqref{eqn:hp}, we have
    \begin{eqnarray*}
        z(f_k) &=& \sum_{i=1}^kz(s\ell_{P/k^i})=\sum_{i=1}^k\left(c_0(s)+\frac{c_1(s)}{k^{in}}\abs{P}+c_2(s)k^{in}\abs{P^\ast}\right)\\
        &=& kc_0(s)+c_1(s)\abs{P}\sum_{i=1}^kk^{-in}+c_2(s)\abs{P^\ast}\sum_{i=1}^kk^{in}\rightarrow0\mbox{ as }k\rightarrow\infty.
    \end{eqnarray*}
    Therefore, $c_2(s)$ has to vanish as the geometric series diverges,
    as well as $c_0(s)$, for every $s\in\R$.

    Now, let's further determine $c_1$ by two different examples.

    \textbf{Example 2.}
    For each function $f$ with $\lim_{x\rightarrow0}f(x)=\infty$,
    let $P\in\poly{n}$ and $P_k=P\left(k^p/f(1/k)\right)^\frac{1}{n}$, for $k=1,2,\ldots$.
    Then, we have
    \[\normk{p}{\ell_{P_k}/k}^p=c_{n,p}k^{-p}\abs{P_k}=c_{n,p}\abs{P}/f(1/k)\rightarrow0 \mbox{ as } k\rightarrow\infty\]
    and
    \[\normk{p}{\nabla\ell_{P_k}/k}^p=\frac{1}{n}k^{-p}S_p(P_k)=\frac{1}{n}S_p(P)k^{-\frac{p^2}{n}}\left(f(1/k)\right)^\frac{p-n}{n}\rightarrow0 \mbox{ as } k\rightarrow\infty.\]
    Thus, $\ell_{P_k}/k\rightarrow0$ in $\sobo{1}{p}$ as $k\rightarrow\infty$.

    By \eqref{eqn:hp}, we obtain
    \[z\left(\ell_{P_k}/k\right)=c_1(1/k)k^p\abs{P}/f(1/k)\rightarrow0 \mbox{ as } k\rightarrow\infty.\]
    Therefore, $c_1(1/k)\sim o\left(f(1/k)/k^p\right)$ as $k\rightarrow\infty$.
    Similarly, considering $-\ell_{P_k}/k$, we obtain the same estimate as $x\rightarrow0^-$.
    Hence, $c_1(x)\sim o(x^pf(x))$ as $x\rightarrow0$.
    It follows that $c_1(x)\sim O(x^p)$ as $x\rightarrow0$ via Remark \ref{thm:bigo}.

    \textbf{Example 3.}
    For each function $f$ with $\lim_{x\rightarrow\infty}f(x)=\infty$,
    let $P\in\poly{n}$ and $P_k=P\big/\left(k^{p^\ast}f(k)\right)^\frac{1}{n}$, for $k=1,2,\ldots$.
    Then, we have
    \[\normk{p}{k\ell_{P_k}}^p=c_{n,p}k^p\abs{P_k}=c_{n,p}k^{p-p^\ast}\left(f(k)\right)^{-1}\abs{P}\rightarrow0 \mbox{ as } k\rightarrow\infty\]
    and
    \[\normk{p}{\nabla k\ell_{P_k}}^p=\frac{1}{n}k^pS_p(P_k)=\frac{1}{n}S_p(P)\left(f(k)\right)^\frac{p-n}{n}\rightarrow0 \mbox{ as } k\rightarrow\infty.\]
    Thus, $k\ell_{P_k}\rightarrow0$ in $\sobo{1}{p}$ as $k\rightarrow\infty$.

    By \eqref{eqn:hp}, we obtain
    \[z\left(k\ell_{P_k}\right)=c_1\left(k\right)k^{-p^\ast}\left(f(k)\right)^{-1}\abs{P}\rightarrow0 \mbox{ as } k\rightarrow\infty.\]
    Therefore, $c_1\left(k\right)\sim o(k^{p^\ast}f(k))$ as $k\rightarrow\infty$.
    Similarly, considering $-k\ell_{P_k}$, we obtain the same estimate as $x\rightarrow-\infty$.
    Hence, $c_1(x)\sim o\left(x^{p^\ast}f(x)\right)$ as $x\rightarrow\infty$.
    It follows that $c_1(x)\sim O(x^{p^\ast})$ as $x\rightarrow\infty$ via Remark \ref{thm:bigo}.
\end{proof}


Now we are ready to prove the result on homogeneous valuations.
\begin{proof}[Proof of Theorem \ref{thm:homo}]
    The backwards direction has already been shown in Theorem \ref{thm:phi}.

    We now consider the forward direction.
    In the light of Lemma \ref{thm:L1p}, it suffices to consider the case
    $f=s\ell_P$ for every $s\in\R$ and $\ell_P\in\pap{n}$.
    In this case, due to Lemma \ref{thm:P1p}, there exists a continuous function $c:\R\rightarrow\R$
    satisfying \eqref{eqn:gc} such that
    \begin{equation}\label{eqn:hp1}
        z(s\ell_P)=c(s)\abs{P}
    \end{equation}
    for every $s\in\R$ and $\ell_P\in\pap{n}$.
    On the other hand, by homogeneity, there exists a constant $c\in\R$ such that
    \begin{equation}\label{eqn:hp2}
        z(s\ell_P)=c\abs{s}^q\abs{P}
    \end{equation}
    for every $s\in\R$ and $\ell_P\in\pap{n}$.
    Formulas \eqref{eqn:hp1} and \eqref{eqn:hp2} yield
    \begin{equation}\label{eqn:hp3}
        c(s)=c\abs{s}^q
    \end{equation}
    for every $s\in\R$.

    For $q<p$ or $q>p^\ast$, since $c(s)$ satisfies \eqref{eqn:gc},
    which is impossible with the expression \eqref{eqn:hp3},
    we have $c=0$. It follows that $z(s\ell_P)=0$
    for every $s\in\R$ and $\ell_P\in\pap{n}$.

    For $p\leq q\leq p^\ast$, set $\tilde{c}={{n+q} \choose q}c$.
    By properties of the beta and the gamma function
    and the layer cake representation, we have
    \begin{eqnarray*}
        c(s) &=& \tilde{c}\abs{s}^q{{n+q} \choose q}^{-1}\\
        &=& \tilde{c}q\abs{s}^q\frac{\Gamma(q)\Gamma(n+1)}{\Gamma(n+q+1)}\\
        &=& \tilde{c}q\abs{s}^q\int_0^1t^{q-1}(1-t)^ndt\\
        &=& \tilde{c}q\int_0^1(\abs{s}t)^{q-1}(1-t)^nd\abs{s}t\\
        &=& \tilde{c}q\int_0^{\abs{s}}t^{q-1}\left(\frac{\abs{s}-t}{\abs{s}}\right)^ndt.
    \end{eqnarray*}
    Thus,
    \begin{eqnarray*}
        c(s)\abs{P} &=& \tilde{c}q\int_0^{\abs{s}}t^{q-1}\abs{\set{\abs{s}\ell_P>t}}dt\\
        &=& \tilde{c}\int_{\R^n}(\abs{s}\ell_P(x))^qdx\\
        &=& \tilde{c}\normk{q}{s\ell_P}^q.
    \end{eqnarray*}
\end{proof}


\section{A more general characterization}\label{sec:final}

We finish the proof of Theorem \ref{thm:main} by the following crucial representation.

\begin{lem}\label{thm:FLC}
    Let the functional $z:\lpa\rightarrow\R$ satisfy $z(0)=0$
    and let $s\mapsto z(sf)$ belong to $\B$ for $s\in\R$ and $\ell_P\in\pap{n}$.
    If there exists a continuous function
    $c:\R\rightarrow\R$ satisfying \eqref{eqn:gc} such that
    \[z(s\ell_P)=c(s)\abs{P},\]
    for every $s\in\R$ and $\ell_P\in\pap{n}$,
    then there exists a continuous function $h\in\G$ such that
    \[z(s\ell_P)=\int_{\R^n}h\circ(s\ell_P).\]
\end{lem}
\begin{proof}
    It suffices to consider the case $s>0$.
    Since there exists a continuous function $c:\R\rightarrow\R$ satisfying \eqref{eqn:gc}, such that
    \[z(s\ell_P)=c(s)\abs{P},\]
    we have
    \[Dz_{\ell_P}(s)=c'(s)\abs{P}.\]
    It follows that $c(s)$ is continuously differentiable in the usual sense.
    Hence $c(s)\in C^{n}(\R)$, due to $s\mapsto z(sf)$ belonging to $C^n(\R)$
    for every $s\in\R$ and $f\in\pap{n}$. Moreover,
    \begin{equation}\label{eqn:diff}
        D^\a z_{\ell_P}(s)=c^{(\a)}(s)\abs{P},
    \end{equation}
    for every non-negative integer $\a\leq n$ and $c^{(n)}\in\lbv$.

    Now, let
    \begin{equation}\label{eqn:int}
        h(s)=\sum_{j=0}^n\frac{1}{j!}{n\choose j}s^jc^{(j)}(s).
    \end{equation}
    We show by induction that there exists a signed measure $\nu$ on $\R$ such that
    \[c(s)=\int_0^s\left(\frac{s-t}{s}\right)^nd\nu(t).\]
    Since $c\in C^n(\R)$ and $c^{(n)}\in\lbv$,
    there exists a signed measure $\nu$ on $\R$ such that
    $h(s)=\nu([0,s))$ for every $s\geq0$.
    Let $h_1(s)=\int_0^sh(x)dx$.
    Then, by Fubini's theorem, we obtain
    \begin{eqnarray*}
        h_1(s) &=& \int_0^s\int_0^xd\nu(t)dx\\
        &=& \int_0^s\int_t^sdxd\nu(t)\\
        &=& \int_0^s(s-t)d\nu(t).
    \end{eqnarray*}
    Let $k\geq2$ and $h_k(s)=\int_0^sh_{k-1}(x)dx$.
    Assume $h_k(x)=\frac{1}{k!}\int_0^x(x-t)^kd\nu(t)$.
    Applying Fubini's theorem again gives
    \begin{eqnarray*}
        h_{k+1}(s) &=& \frac{1}{k!}\int_0^s\int_0^x(x-t)^kd\nu(t)dx\\
        &=& \frac{1}{k!}\int_0^s\int_t^s(x-t)^kdxd\nu(t)\\
        &=& \frac{1}{(k+1)!}\int_0^s(s-t)^{k+1}d\nu(t).
    \end{eqnarray*}
    Thus, in particular, we have
    \[h_n(s)=\frac{1}{n!}\int_0^s(s-t)^nd\nu(t).\]
    On the other hand, by \eqref{eqn:int}, we have
    \begin{eqnarray*}
        h(x) &=& c(x)+\frac{1}{n!}x^nc^{(n)}(x)+
        \sum_{j=1}^{n-1}\frac{1}{j!}\left({{n-1}\choose j}+{{n-1}\choose{j-1}}\right)x^jc^{(j)}(x)\\
        &=& \sum_{j=0}^{n-1}\frac{1}{j!}{{n-1}\choose j}x^jc^{(j)}(x)
        +\sum_{j=0}^{n-1}\frac{1}{(j+1)!}{{n-1}\choose j}x^{j+1}c^{(j+1)}(x)\\
        &=& \sum_{j=0}^{n-1}\frac{1}{(j+1)!}{{n-1}\choose j}\left(x^{j+1}c^{(j)}(x)\right)'.
    \end{eqnarray*}
    Hence,
    \[h_1(s)=\int_0^sh(x)dx=\sum_{j=0}^{n-1}\frac{1}{(j+1)!}{{n-1}\choose j}s^{j+1}c^{(j)}(s).\]
    Assume that $h_k(x)=\sum_{j=0}^{n-k}\frac{1}{(j+k)!}{{n-k}\choose j}x^{j+k}c^{(j)}(x)$.
    Similarly, we obtain
    \begin{eqnarray*}
        h_k(x) &=& \frac{1}{k!}x^kc(x)+\frac{1}{n!}x^nc^{(n-k)}(x)\\
        &&+\sum_{j=1}^{n-k-1}\frac{1}{(j+k)!}\left({{n-k-1}\choose j}+{{n-k-1}\choose{j-1}}\right)x^{j+k}c^{(j)}(x)\\
        &=& \sum_{j=0}^{n-k-1}\frac{1}{(j+k)!}{{n-k-1}\choose j}x^{j+k}c^{(j)}(x)\\
        &&+\sum_{j=0}^{n-k-1}\frac{1}{(j+k+1)!}{{n-k-1}\choose j}x^{j+k+1}c^{(j+1)}(x)\\
        &=& \sum_{j=0}^{n-k-1}\frac{1}{(j+k+1)!}{{n-k-1}\choose j}\left(x^{j+k+1}c^{(j)}(x)\right)'.
    \end{eqnarray*}
    It follows that
    \[h_{k+1}(s)=\int_0^sh_k(x)dx=\sum_{j=0}^{n-(k+1)}\frac{1}{(j+k+1)!}{{n-(k+1)}\choose j}s^{j+k+1}c^{(j)}(s).\]
    Thus, in particular, we have $h_n(s)=\frac{1}{n!}s^nc(s)$.
    Therefore, by the layer cake representation, we have
    \begin{eqnarray*}
        z(s\ell_P) = c(s)\abs{P} &=& \int_0^s\left(\frac{s-t}{s}\right)^n\abs{P}d\nu(t)\\
        &=& \int_0^s\abs{\set{s\ell_P>t}}d\nu(t)\\
        &=& \int_{\R^n}h\circ(s\ell_P).
    \end{eqnarray*}

    Furthermore, for fixed $P\in\poly{n}$,
    \[s^k\Dzlp{k}=s^kc^{(k)}(s)\abs{P}\]
    satisfies \eqref{eqn:gc} for every integer $0\leq k\leq n$.
    Therefore, as defined in \eqref{eqn:int},
    $h$ also satisfies \eqref{eqn:gc}.
\end{proof}

Theorem \ref{thm:main} follows as an immediate corollary of
Theorem \ref{thm:phi} and Lemmas \ref{thm:L1p}, \ref{thm:P1p}, and \ref{thm:FLC}.

\bigskip

\textbf{Acknowledgement.}
The author wishes to thank the referees for valuable suggestions
and careful reading of the original manuscript.
The work of the author was supported in part by Austrian Science Fund (FWF) Project P23639-N18
and the National Natural Science Foundation of China Grant No. 11371239.



\end{document}